\documentclass[12pt,draftcls,onecolumn]{IEEEtran}
\usepackage[dvips]{graphicx}
\usepackage{amssymb}
\usepackage{mathrsfs}
\input{epsf.sty}
\usepackage{epsfig}
\usepackage[dvipsnames,usenames]{color}
\usepackage{amsmath}
\usepackage{amssymb}
\usepackage{graphicx}
\usepackage{times}
\usepackage{graphics,color}
\usepackage{subfigure}

\newtheorem{theorem}{Theorem}
\newtheorem{lemma}{Lemma}

\newtheorem{corollary}{Corollary}
\newtheorem{definition}{Definition}
\newtheorem{remark}{Remark}

\newcommand{\oneton}{1,\cdots,n}
\newcommand{\onetoK}{1,\cdots,K}

\begin{document}
\title
{Achieving Cluster Consensus in Continuous-Time Networks of Multi-Agents
With Inter-Cluster Non-Identical Inputs }
\author{Yujuan Han, Wenlian
Lu,~\IEEEmembership{Member,~IEEE}, Tianping
Chen,~~\IEEEmembership{Senior~Member,~IEEE}
\thanks{Y. Han, W. Lu, T. Chen are with School of Mathematical Sciences, Fudan University, No. 220 Handan Lu, Shanghai 200433, China.
Email addresses: 09210180039@fudan.edu.cn (Y. Han), wenlian@fudan.edu.cn (W. Lu),
tchen@fudan.edu.cn (T. Chen).}
\thanks{This work is jointly supported by the National Key Basic Research and
Development Program (No. 2010CB731403), the National Natural Sciences
Foundation of China under Grant (Nos. 61273211 and 61273309), the Marie
Curie International Incoming Fellowship from the European Commission
(FP7-PEOPLE-2011-IIF-302421), the Program for New Century Excellent Talents
in University (NCET-13-0139), and also sponsored by the Laboratory of
Mathematics for Nonlinear Science, Fudan University.}}
\date{}

\maketitle
\begin{abstract}
In this paper, cluster consensus in continuous-time networks of
multi-agents with time-varying topologies via non-identical inter-cluster
inputs is studied. The cluster consensus contains two aspects:
intra-cluster synchronization, that the state differences between agents in
the same cluster converge to zero, and inter-cluster separation, that the
states of the agents in different clusters do not approach.
$\delta$-cluster-spanning-tree in continuous-time networks of multi-agent
systems plays essential role in analysis of cluster synchronization.
Inter-cluster separation can be realized by imposing adaptive inputs that
are identical within the same cluster but different in different clusters,
under the inter-cluster common influence condition. Simulation examples
demonstrate the effectiveness of the derived theoretical results.

\end{abstract}
\section{ Introduction}
Consensus problems of multi-agent systems have attracted broad
attentions from various contexts (see \cite{Reynolds87}-\cite{Jadbabaie}).
In general, the main objective of consensus problems is to make all
agents converge to some common state by designing proper algorithms.
For this purpose, various consensus algorithms have been proposed
\cite{Olfati-Saber}-\cite{Liu Xi}.

The results of almost all previous works were concerned with consensus with
a common consistent state, while we are considering cluster consensus,
i.e., agents in networks are divided into several disjoint groups, called
{\em clusters}, in the sense that all agents in the same cluster completely
synchronize but the dynamics in different clusters does not coincide. In
reality, a number of practical models can be transformed into this cluster
consensus problem, for instance, social learning network under different
environments \cite{discrete}. Social learning focuses on the opinion
dynamics in the society, in which individuals engage in communication with
their neighbors in order to learn from their experiences. Consider that the
belief of each individual is affected by different religious beliefs or
cultural backgrounds. This affection flags the clusters that each
individual belongs to. 


In \cite{Yu}-\cite{YChen}, the authors considered cluster (group)
synchronization (consensus) problems of networks with multi-agents. In
\cite{Yu,Tan}, for linearly coupled multi-agents systems,  the authors
derived conditions on coupling matrix to guarantee group
consensus(intra-cluster synchronization) , but the inter-cluster separation
was not considered. In \cite{Lu10}, agents in different clusters have different dynamics of uncoupled node systems, the inter-cluster separation was not proved rigorously (but only assumed). Since  it is quite difficult
to prove inter-cluster separation for general nonlinear coupled systems (up to now, no way to prove).  In \cite{Xia}, the dynamics of nodes are special, hence,
the final states of agents can be given directly. In this paper, In this paper, the inter-cluster separation is actually one of main aims, which is realized by imposing the inter-cluster different, intra-cluster identical inputs.

In our previous paper \cite{discrete}, we
investigated cluster consensus problem in discrete-time networks of
multi-agents, which provided the basic ideas. 
However, There still is big difference between discrete-time networks and
continuous-time system. In addition, in comparison with \cite{discrete}, in
the present paper, the static inter-cluster influence matrix in
\cite{discrete} is replaced by time-varying inter-cluster influence matrix
sequence; the assumption of existence of self-links in \cite{discrete} are
removed; the formation of inputs to a more general scenario are extended,
while \cite{discrete} considered that the inputs among different clusters
only differ by a proportionality constant. Finally, the concepts relating
graph theory are generalized, too. For example, we propose
"$\delta$-cluster-spanning-tree across time interval $\mathcal{I}$" (see
below). 

\section{ Preliminaries}

In this section, we present some necessary notations and
definitions of graph and matrix theory. For more details,
we refer readers to textbooks \cite{God,Horn}.


For a matrix $L$, denote $L_{ij}$ the element of $L$ on the $i$-th
row and $j$-th column. $L^{\top}$ denotes the transpose of $L$. $E_n$
and $O_n$ denote the $n$-dimensional identity matrix and zero matrix. ${\bf 1}$ denotes the column vector whose components
all equal to $1$ and ${\bf 0}$ denotes the column vector whose components
all equal to $0$.
$\|z\|$ denotes a vector norm of a vector $z$ and
$\|L\|$ denotes the matrix norm of $L$ induced by the vector norm
$\|\cdot\|$.

An $n\times n$ matrix $A$ is called a {\em stochastic matrix} if
$A_{ij}\ge 0$ for all $i$, $j$, and $\sum_{j=1}^{n}A_{ij}=1$ for
$i=1,\cdots,n$. An $n\times n$ matrix $L$
is called a {\em Metzler matrix with zero row sums} if $L_{ij}\ge 0$
and $\sum_{j=1}^{n}L_{ij}=0$ holds for all $i\ne j$, $i=\oneton$.

A directed graph $\mathcal{G}=\{\mathcal V, \mathcal E\}$ consists of a
vertex set $\mathcal{V}=\{v_{1},\cdots,v_{n}\}$, a directed edge set
$\mathcal{E}\subseteq \mathcal{V}\times \mathcal{V}$, i.e., an edge is an
ordered pair of vertices in $\mathcal{V}$.
A (directed) {\em path} of length $l$ from vertex $v_{j}$ to $v_{i}$, denoted
by $(v_{r_{1}},\cdots,v_{r_{l+1}})$, is a sequence of $l+1$ distinct
vertices with $v_{r_{1}}=v_{i}$ and $v_{r_{l+1}}=v_{j}$ such that
$(v_{r_{k}},v_{r_{k+1}})\in \mathcal{E}$.
We say that $\mathcal G$ has self-links if
$(v_{i},v_{i})\in\mathcal E$ for all $v_{i}\in\mathcal V$.

An $n \times n$ nonnegative matrix $A$ can be
associated with a directed graph
$\mathcal{G}(A)$ in such a way that
$(v_{i},v_{j})\in \mathcal{E}(\mathcal{G}(A))$ if and only if
$A_{ij}>0$. Similarly, for a Metzler matrix $L$, it is associated with a graph without
self-links, denoted by $\mathcal{G}(L)$.

\begin{definition}\cite{discrete}\label{cluster}
For a graph $\mathcal G=(\mathcal V,\mathcal E)$,  a
$\emph{clustering}$ $\mathcal{C}$ is defined as a disjoint division
of the vertex set, namely, a sequence of subsets of $\mathcal{V}$,
$\mathcal{C}=\{\mathcal{C}_{1},\cdots,\mathcal{C}_{K}\}$, that
satisfies: (1) $\bigcup_{p=1}^{K}\mathcal{C}_{p}=\mathcal{V}$; (2)
 $\mathcal{C}_{k}\bigcap\mathcal{C}_{l}=\emptyset$, $k\neq l$.
\end{definition}

Consider the following continuous-time system with
external adapted inputs:
\begin{eqnarray}
\dot{x}_{i}(t)=\sum_{j=1}^{n}L_{ij}(t)[x_{j}(t)-x_{i}(t)]+I_{i}(t),~~~i=1,\cdots,n
\label{time-varying1}
\end{eqnarray}
where $t\in \mathbb{R}^{+}=[0,\infty)$ and $x_{i}(t)\in \mathbb{R}$
denotes the state variable of the agent $i$, $L_{ij}(t)\geq 0$ denotes
the coupling weight from agent $j$ to $i$, $I_{i}(t)$, $i=1,\cdots,n$
are external scalar inputs. 
Let $L_{ii}(t)=-\sum_{j=1,j\neq i}^{n}L_{ij}(t)$, then for each $t>0$, the
connection matrix $L(t)=[L_{ij}(t)]_{i,j=1}^{n}$ is a Metzler matrix with
zero row sum.
The matrix $L(t)$ is associated with a time-varying graph $\mathcal{G}(L(t))$.

For systems with switching topologies, some researchers introduce the
concept of dwell time, which is a pre-specified positive constant to
describe the time length staying in current topology, i.e., in some time
interval $[t_{1},t_{2}]$, $L(t)=L$ are constant. In this paper, we don't
make this assumption.  By using the concept of $\delta$-edge \cite{Moreau},
we transform the continuous-time case to the discrete case with some
sophisticated analysis.

\begin{definition}\label{delta-cluster spanning}
$\mathcal G(L(t))$ is said to have a $\delta$-edge from vertex $v_{j}$ to
$v_{i}$ across $[t_{1},t_{2})$, if
$\int_{t_{1}}^{t_{2}}L_{ij}(t)dt>\delta$. For a given clustering
$\mathcal{C}=\{\mathcal{C}_{1},\cdots,\mathcal{C}_{K}\}$, $L(t)$ has a
$\delta$-cluster-spanning-tree across $[t_{1},t_{2})$ (w.r.t.
$\mathcal{C}$) if each cluster $\mathcal{C}_{p}$, $p=\onetoK$, has a vertex
$v_{p}\in\mathcal V$ and a $\delta$-path (path composed of $\delta$-edges)
from $v_{p}$ to all vertices in $\mathcal{C}_{p}$ across $[t_{1},t_{2})$.
\end{definition}

It should be pointed out that the root of $\mathcal C_{p}$ and the paths
from the root to the vertices in $\mathcal C_{p}$ do not necessarily in
$\mathcal C_{p}$; the root vertex of a cluster is unnecessarily identical
with roots in other clusters.

\begin{definition}\label{cluster scrambling}
For a given clustering
$\mathcal{C}=\{\mathcal{C}_{1},\cdots,\mathcal{C}_{K}\}$, we say
$\mathcal{G}$ is {\em cluster-scrambling} (w.r.t. $\mathcal C$) if for
any pair of vertices $(v_{p_{1}},v_{p_{2}})\subset\mathcal C_{p}$,
there exists a vertex $v_{k}\in \mathcal V$, such that both
$(v_{k},v_{p_{1}})$ and $(v_{k},v_{p_{2}})$ are in $\mathcal E$.
\end{definition}

In \cite{discrete}, we extended ergodicity coefficient \cite{Shen} and
Hajnal diameter \cite{Hajnal} to the clustering case and defined the
cluster ergodicity coefficient (w.r.t $\mathcal C$) of a stochastic matrix
$A$ as
\begin{eqnarray*}
\mu_{\mathcal C}(A)=\min_{p=1,\cdots,K}\min_{i,j\in\mathcal{C}_{p}}\sum_{k=1}^{N}{\min(A_{ik},A_{jk})}
\end{eqnarray*}
It can be seen that $\mu_{\mathcal C}(A)\in [0,1]$ and $A$ is
cluster-scrambling (w.r.t. $\mathcal C$) if and only if $\mu_{\mathcal
C}(A)>0$. Furthermore, we say $A$ is {\em $\delta$-cluster-scrambling} if
$\mu_{\mathcal C}(A)>\delta$.

Hajnal diameter proposed in \cite{Hajnal} was also generalized to the cluster
case:
\begin{definition}\cite{discrete}
For a given clustering $\mathcal C$ and a matrix $A$, which has row
vectors $A_{1},A_{2},\cdots,A_{n}$, define the cluster Hajnal
diameter as $
\Delta_{\mathcal C}(A)=\max_{p=1,\cdots,K}\max_{i,j\in C_{p}}\|A_{i}-A_{j}\|
$
for some norm $\|\cdot\|$.
\end{definition}
\begin{remark}
In \cite{discrete}, we have generalized Hajnal inequality to the
following cluster Hajnal inequality, i.e.
\begin{eqnarray}
\Delta_{\mathcal{C}}(AB)\leq
(1-\mu_{\mathcal{C}}(A))\Delta_{\mathcal{C}}(B)\label{chi}
\end{eqnarray}
where $A$ is a stochastic matrix and $B$ is a matrix or a vector.
\end{remark}

This inequality indicates that the cluster Hajnal diameter of $AB$ strictly
decreases when compared with $B$, if $A$ is cluster scrambling, i.e.,
$\mu_{\mathcal{C}}(A)>0$.

\section{ Cluster consensus analysis}

Let $x(t)=[x_{1}(t),\cdots,x_{n}(t)]^{\top}\in\mathbb R^{n}$ denote
the state trajectory of all agents and $I(t)=[I_{1}(t),\cdots,I_{n}(t)]^{\top}$.
The system (\ref{time-varying1}) can be written in the following
impact form:
\begin{eqnarray}\label{time-varying2}
\dot{x}(t)=L(t)x(t)+I(t)
\end{eqnarray}
\begin{definition}
System (\ref{time-varying2}) is said to be {\em
intra-cluster synchronized} if any solution $x(t)$ satisfies
$\lim_{t\rightarrow\infty}|x_{i}(t)-x_{i'}(t)|=0$  for all $ i,i'\in
\mathcal{C}_{p}$ and $p=1,\cdots,K$; {\em inter-cluster
separated} if
$\limsup_{t\rightarrow\infty}\min_{i\in\mathcal{C}_{k},j\in\mathcal{C}_{l},k\neq
l}|x_{i}(t)-x_{j}(t)|>0$. The system (\ref{time-varying1}) realizes {\em
cluster consensus} if each solution $x(t)$ is bounded, intra-cluster
synchronized and inter-cluster separated.
\end{definition}


It can be seen that intra-cluster synchronization is equivalent to the
stability of the following {\em cluster consensus subspace} w.r.t. the
clustering $\mathcal C$:
\begin{eqnarray*}
\mathcal S_{\mathcal C}=\bigg\{x\in\mathbb R^{n}:~x_{i}=x_{j}, {\rm~if~}
i,j\in \mathcal C_{p},  ~p=\onetoK\bigg\}
\end{eqnarray*}

A prerequisite requirement for cluster consensus is that
$\mathcal{S}_{\mathcal{C}}$ should be invariant through
(\ref{time-varying1}).
\begin{lemma}\label{invariant}
If the following conditions are satisfied: (1). $I_{i}(t)=I_{j}(t)$ for all
$i,j\in\mathcal C_{p}$ and all $p=\onetoK$; (2). for each pair of $p$ and
$q$, $\sum_{j\in C_{q}}L_{ij}(t)$ is identical w.r.t. all $i\in \mathcal
C_{p}$ at any time $t$, then the cluster-consensus subspace is invariant
through (\ref{time-varying1}).
\end{lemma}

The proof is similar to Lemma 3 in \cite{discrete} and is omitted.

The input is said to be {\em intra-cluster identical} if the condition (1)
in Lemma \ref{invariant} is satisfied, and the matrix $L(t)$ has {\em
inter-cluster common influence} if condition (2) is satisfied.

Denote $B_{pq}(t)\triangleq\sum_{j\in C_{q}}L_{ij}(t)$ w.r.t. all $i\in
\mathcal C_{p}$ at any time $t$ and call $B(t)=[B_{pq}(t)]$ {\em the
inter-cluster common influence matrix}.
\subsection{ Theoretical results}


In the following, we assume
\begin{itemize}
\item ${\bf \mathcal A}_{1}$: For any $t\ge t_{0}$, $L(t)$ is Metzler matrix
with all row sums zeros and the elements $L_{ij}(t)\geq 0$ are piecewise continuous;
\item
${
\bf \mathcal A}_{2}$ ({\em inter-cluster common influence}): For any
$t\ge t_{0}$, there exists a zero row sum Metzler matrix
$B(t)=[B_{p,q}(t)]_{p,q=1}^{K}\in R^{K,K}$, where
\begin{eqnarray}
\sum_{j\in C_{q}}L_{ij}(t)=B_{p,q}(t),~i\in\mathcal C_{p},~p,q=\onetoK\label{B1}
\end{eqnarray}
\end{itemize}

We highlight that the concept {\em inter-cluster common influence }
coincides with the concept {\em odrinary lumpability} in Markov chain
theory \cite{Buchholz}.


$\mathcal A_3$: For any $i$, $I_i(t)$ is piecewise continuous, both $I_i(t)$ and $\int_{t_0}^{t}I_i(s)ds$ are bounded, and
$
I_{i}(t)=I_{j}(t)\triangleq\tilde{I}_p(t)$, for all $i,j\in\mathcal C_{p}$, $p=\onetoK$.
Let $\tilde{I}(t)=[\tilde{I}_1(t),\cdots,\tilde{I}_K(t)]^{\top}$.

\begin{remark}
In this paper, we focus on finding the simplest external inputs to
guarantee the intra-cluster synchronization and inter-cluster separation.
Here the inputs are intra-cluster identical, which
counts for intra-cluster synchronization, and inter-cluster different and state-independent, which
counts for the inter-cluster separation
\end{remark}
\begin{remark}
If the linearly coupled system can intra-cluster synchronize,  the external inputs proposed in this paper can always be used to guarantee the inter-cluster separation, which implies cluster consensus of the linearly coupled systems.
\end{remark}
\begin{lemma}\label{delta-scrambling}
Suppose $\Phi(t,t_{0})$ is the basic solution
matrix of the homogeneous system:
\begin{eqnarray}\label{homogeneous}
\dot{v}(t)=L(t)v(t)
\end{eqnarray}
where $L(t)$ satisfies $\mathcal {A}_{1},\mathcal A_2$. Then,
(1). $\Phi(t,t_{0})$ is a stochastic matrix;
 (2). 
If $L(t)$ has a $\delta$-cluster-spanning-tree across time
interval $[t_0,t_1)$ and $\int_{t_0}^{t_1}L_{ij}(s)ds<M_1$
holds for all $i\neq j$ and some $M_{1}> 0$, then $\Phi(t_1,t_0)$ has
a $\delta_1$-cluster-spanning-tree,
where $\delta_1=\min\{1,\delta\}e^{-(n-1)M_1}$.
\end{lemma}
{\em Proof.} 1). Denote $\Phi(t,t_{0})=[\Phi_{ij}(t,t_{0})]\in\mathbb R^{n\times
n}$. Since $L(t)$ satisfies assumption $\mathcal{A}_{2}$, if $x(t_{0})=\textbf{1}_{n}$, then the solution
must be $x(t)=\textbf{1}_{n}$, which implies each row sum of
$\Phi(t,t_{0})$ equals 1. Next, we will prove all elements in
$\Phi(t,t_{0})$ are nonnegative. Note that the $i$-th column of
$\Phi(t,t_{0})$ can be regarded as the solution of the following
equation:
\begin{eqnarray}\label{initial}
\begin{cases}\dot{x}(t)=L(t)x(t)\\
x(t_{0})=e_{i}^{n} \end{cases}
\end{eqnarray}
here $e_{i}^{n}$ is an $n$-dimensional vector whose $i$-th component
is 1 and all the other components are zero. For any $t>t_{0}$, if
$i_{0}=i_{0}(t)$ is the index with
$x_{i_{0}}(t)=\min_{j=1,\cdots,n}x_{j}(t)$, then
$\dot{x}_{i_{0}}(t)=\sum_{j\neq
i_{0}}L_{i_{0}j}(x_{j}(t)-x_{i_{0}}(t))\geq 0$. This implies that
$\min_{j}x_{j}(t)$ is always nondecreasing for all $t>t_{0}$.
Therefore, $x(t)\geq 0$ holds for $t\geq
t_{0}$. Therefore, $\Phi(t,t_{0})$ is a
stochastic matrix.

2). Consider system (\ref{initial}), since $x_{j}(t)\geq 0$ holds
for all $j=1,\cdots,n$, so $\dot{x}_{i}(t)\geq L_{ii}(t)x_{i}(t)$,
and $x_{i}(t)\geq e^{\int_{t_{0}}^{t}L_{ii}(s)ds}\ge e^{-(n-1)M_1}$.
Meanwhile,
we can conclude that $\Phi_{ii}(t_1,t_0)$ is positive. For each
$k\neq i$,
\begin{eqnarray*}
x_{k}(t)&=&\sum_{j\neq k}\int_{t_{0}}^{t}e^{\int_{\tau}^{t}L_{kk}(s)ds}L_{kj}(\tau)x_{j}(\tau)d\tau\\
&\geq& \int_{t_{0}}^{t}e^{\int_{\tau}^{t}L_{kk}(s)ds}L_{ki}(\tau)x_{i}(\tau)d\tau\\
&\geq& e^{-(n-1)M_1}\int_{t_{0}}^{t}L_{ki}(\tau)d\tau
\end{eqnarray*}
So, if $L(t)$ has a $\delta$-edge from vertex $j$ to vertex $i$
across $[t_{0},t_{1}]$, then $\Phi_{ij}(t_{1},t_{0})\geq e^{-(n-1)M_1}\delta$, which means
$\Phi(t_1,t_0)$ has a $\delta_1$-cluster-spanning-tree

We also present the following assumption for $L(t)$:

$\mathcal A_4$: There exist an infinite time interval sequence $[t_0,t_1), [t_2,t_3),\cdots, [t_{2n},t_{2n+1}), \cdots$, where $t_0<t_1\leq t_2<t_3\leq\cdots$ and
a positive sequence $\{\delta_{k}\}$ which satisfies
$\sum_{k=1}^{+\infty}(\delta_{k})^{n-1}=+\infty$. And for any $[t_{2k},t_{2k+1})$, there is a division: $t_{2k}=t_{2k}^0<t_{2k}^1<\cdots < t_{2k}^{n-1}=t_{2k+1}$, such that $L(t)$ has a
$\delta_k$-cluster-spanning-tree across
$[t_{2k}^m,t_{2k}^{m+1})$ and
$\int_{t_{2k}^m}^{t_{2k}^{m+1}}L_{ij}(s)ds<M_1,~i\neq j$ with some $M_1>0$, $m=0,\cdots,n-2.$

Then, we have the following theorem.
\begin{theorem}\label{intra-cluster}
Assume that $L(t)$ satisfies assumptions $\mathcal A_{1},\mathcal A_2$ and $\mathcal{A}_4$. If input $I(t)$ satisfies assumption $\mathcal A_3$, then system (\ref{time-varying1}) intra-cluster synchronizes.
\end{theorem}
\begin{proof} Under the assumptions $\mathcal A_{1}, \mathcal A_{3}$, system
(\ref{time-varying1}) has a unique solution for any given initial
value $x(t_{0})$ \cite{Adrianova}, which has the form $
x(t)=\Phi(t,t_{0})x(t_{0})+\int_{t_{0}}^{t}\Phi(t,s)I(s)ds$
with $\Phi(\cdot,\cdot)$ defined in Lemma \ref{delta-scrambling}, which implies
 that
$\Phi(t_{2k}^{i+1},t_{2k}^{i}),i=0,\cdots,n-2$ are stochastic
matrices and have a $\delta_{k}'$-cluster-spanning-tree with
$\delta_{k}'=\min\{1,\delta_k\}e^{-(n-1)M_1}>0$. Lemma 1 in \cite{discrete}
indicates that $\Phi(t_{2k+1},t_{2k})$ is
$\eta_{k}$-cluster-scrambling with $\eta_{k}=(\delta_{k}')^{n-1}$. By
inequality (\ref{chi}), for any $t\in[t_{2n},t_{2n-1})$, we have
$\Delta_{\mathcal{C}}(\Phi(t,t_{0}))\leq \prod_{k=1}^{n}(1-\eta_{k})\Delta_{\mathcal{C}}(E_{n}).
$

The assumption $\sum_{k=1}^{+\infty}\delta_{k}^{n-1}=+\infty$ implies $\sum_{k=1}^{+\infty}\eta_{k}=+\infty$, which is equivalent to $\lim_{n\to\infty}\prod_{k=1}^{n}(1-\eta_{k})=0$. Hence, $\Delta_{\mathcal{C}}(\Phi(t,t_{0}))$ converges to zero as time tends to infinity.
Since $L(t)$ satisfies the inter-cluster common influence condition,
the cluster consensus subspace is an invariant subspace of
$\Phi(t,t_{0})$. Note that $\Delta_{\mathcal
C}(I(t))=0$. Thus $\Delta_{\mathcal C}(\Phi(t,t_{0})I(t))=0$ for all
$t\ge t_{0}$, which means
$\Delta_{\mathcal{C}}(\int_{t_{0}}^{t}\Phi(t,s)I(s)ds)=0$.
Therefore, we have $\Delta_{\mathcal C}(x(t))\le \Delta_{\mathcal
C}(\Phi(t,t_{0}) x(t_{0}))$ converges to zero as $t\to\infty$.
\end{proof}

For any vector $z=[z_1,\cdots,z_K]^{\top}$, define
\begin{eqnarray}\label{eta}
\eta(z)=\min_{i\neq j}|z_i-z_j|
\end{eqnarray}

\begin{theorem}\label{inter-separation}
Assume that $L(t)$ satisfies assumptions $\mathcal A_{1},\mathcal A_2$ and
$\mathcal{A}_4$.  Let $\Psi(t,t_0)$ be the solution matrix of system
$\dot{z}(t)=B(t)z(t)$. If $I(t)$ satisfies assumption $\mathcal A_3$,
$I_i(t)$ does not converge to zero, $i=1,\cdots,n$, and
$\limsup\limits_{t\rightarrow\infty}~
\eta(\int_{t_0}^{t}\Psi(t,s)\tilde{I}(s)ds)\geq \delta'$ with some
$\delta'>0$, then for almost all initials $x(t_0)$, system
(\ref{time-varying1}) reaches cluster consensus.
\end{theorem}
\begin{proof}  We only need to prove that for almost all initials
$x(t_0)$, system reaches inter-cluster separation. We introduce the
Lyapunov exponent of (\ref{homogeneous}) as follows:
\begin{eqnarray*}
\lambda(v)=\overline{\lim_{t\to\infty}}\frac{1}{t}\log\bigg(\|\Phi(t,t_{0})v\|\bigg).
\end{eqnarray*}
From the Pesin's theory \cite{Pesin}, the Lyapunov exponents can
only pick finite values and provide a splitting of $\mathbb R^{n}$.
Namely, there is a subspace direct-sum division: $
\mathbb R^{n}=\oplus_{j=1}^{J}V_{j}$,
and $\lambda_{1}>\cdots>\lambda_{J}$, possibly $J<n$, such that for
each $v\in V_{j}$, $\lambda(v)=\lambda_{j}$. It's clear that
$\lambda_{1}=0$ because $L(t)$ is a Metzler matrix with zero row
sum. Let $V=\oplus_{j>1}V_{j}$.

We make the following claim.

{\em Claim}: $\mathbb R^{n}=\mathcal S_{\mathcal C}+V$. This claim
is proved in the Appendix. Therefore, for any
$x(t_{0})\in\mathbb R^{n}$, we can find a vector $y_{0}\in\mathcal
S_{\mathcal C}$ such that $x(t_{0})-y_{0}\in V$. Suppose $y(t)$ is the solution of
system:
$
\dot{y}(t)=L(t)y(t)+I(t),~y(t_{0})=y_{0}.
$
Letting $\delta x(t)=x(t)-y(t)$, then it satisfies $
\dot{\delta x}(t)=L(t)\delta x(t)$ with $\delta x(t_{0})=y_{0}-x(t_{0})\in V$,
which implies $\lim_{t\to\infty} \delta x(t)=0$, i.e.
$\lim_{t\to\infty}[x(t)-y(t)]=0$.

Thus, instead of $x(t)$, we will discuss whether
$y(t)\in\mathcal{S}_{\mathcal C}$ inter-cluster separate.
Furthermore, we can replace $y(t)$ by a lower-dimensional vector
$\tilde{y}(t)\in R^{K}$ with $\tilde{y}_{p}(t)=y_{i}(t)$
for some $i\in\mathcal C_{p}$.

Then, we
will discuss the following system:
\begin{eqnarray}\label{sys3}
\dot{\tilde{y}}(t)=B(t)\tilde{y}(t)+\tilde{I}(t)\label{sum'}
\end{eqnarray}
where $B(t)$ is defined in assumption ${\bf\mathcal A}_{2}$ and
$\tilde{I}(t)$ is defined in assumption $\mathcal A_3$.
It is well known that the solution of
(\ref{sys3}) can be written as
\begin{eqnarray*}
\tilde{y}(t)&=&\Psi(t,t_0)\tilde{y}(t_{0})+\int_{t_{0}}^{t}\Psi(t,s)\tilde{I}(s)ds
\end{eqnarray*}
Since $\Psi(t,t_0)$ is a stochastic matrix and $\tilde{y}(t_{0})$ is bounded,
we have $Z_{1}(t)=\Psi(t,t_0)\tilde{y}(t_{0})$ is always bounded. Hence, for any time sequence
$\{t_n\}$, $Z_{1}(t_n)$ has a convergent sub-sequence, still denoted by $\{t_{n}\}$.
Let $Z_2(t)=\int_{t_{0}}^{t}\Psi(t,s)\tilde{I}(s)ds$. From the condition
$\limsup_{t\rightarrow\infty}\eta_c(Z_2(t))\geq\delta'$, one can find a time sequence $\{\hat{t}_{i}\}_{i=1}^{\infty}$ such that
$\eta_c(Z_{2}(\hat{t}_{n}))\geq\delta'/2$.
This implies that each pair of components in $Z_{2}(\hat{t}_{n})$ are not identical.
Without loss of generality, suppose
$\lim_{n\rightarrow\infty}Z_{1}(\hat{t}_n)=Z_1^*$, $\lim_{n\rightarrow\infty}Z_{2}(\hat{t}_n)=Z_2^*$; otherwise, we can
choose a sub-sequence of $\{\hat{t}_n\}$ instead. Obviously, $\eta_c(Z_2^*)\geq\frac{\delta'}{2}$. Furthermore, for almost
every initial value $x(t_{0})$, associated with almost every
$\tilde{y}(t_{0})$, $Z_{1}(\hat{t}_n)\tilde{y}(t_{0})+ Z_{2}(\hat{t}_{n})$
has no pair of components identical when $n$ is sufficiently large.
Therefore, for almost every initial value $x(t_{0})$, when
$n$ is sufficiently large, $\tilde{y}(\hat{t}_{n})$ has no identical
components, which implies that the state of one cluster in $y(\hat{t}_{n})$
are not identical to another.
\end{proof}

In the following corollaries, we suppose the inputs among different
clusters differ by proportionality constants,
\begin{eqnarray}\label{special-input}
I_i(t)=\alpha_p u(t), ~~if~~i\in\mathcal C_p
\end{eqnarray}
$\alpha_1,\cdots,\alpha_K$ are constants and $u(t)$ is a scale function.
Let $\tilde{\zeta}=[\alpha_1,\cdots,\alpha_K]^{\top}$. This kind of input
is easy to construct, as we only need to give a scale input $u(t)$ and $\tilde{\zeta}$.
\begin{corollary}
\label{corollary1}
Suppose $L(t)$ satisfies $\mathcal A_1,
\mathcal A_{2}, \mathcal A_4$ and $I(t)$ has form (\ref{special-input}) with $\mathcal
A_3$. Let $\Psi(t,t_0)$ be the solution matrix of $\dot{z}(t)=B(t)z(t)$.
If $u(t)$ does not converge to zero and $\limsup\limits_{t\to\infty}\\{~\rm
rank}(\int_{t_0}^t\Psi(t,s)u(s)ds)=K$, then for almost all initials
$x(t_0)$ and $\tilde{\varsigma}$, system
(\ref{time-varying1}) can cluster consensus.
\end{corollary}
\begin{proof} Let $Z_3(t)=\int_{t_0}^{t}\Psi(t,s)u(s)ds$. From the
assumption $\limsup\limits_{t\to\infty}{~\rm
rank}(\int_{t_0}^t\Psi(t,s)u(s)ds)=K$, one can find a time sequence
$\{\hat{t}_n\}_{n=1}^{\infty}$ such that
$\lim_{n\to\infty}Z_3(\hat{t}_n)=Z_3^*$ and ${\rm rank}(Z_3^*)=K$.
Hence, the set $\{\tilde{\zeta}|there~exist~ i,j,~such~that~
[Z_3^*\tilde{\zeta}]_i=[Z_3^*\tilde{\zeta}]_j\}$ is of zero measure
in $\mathbb R^K$, which means that for almost every
$\tilde{\zeta}\in\mathbb R^K$, each pair of components in
$Z_3^*\tilde{\zeta}$ are not identical, i.e.
$\eta(Z_3^*\tilde{\zeta})\geq 2\delta'$ with some $\delta'>0$.
Therefore, all conditions in Theorem \ref{inter-separation} hold.
\end{proof}

In the following corollary, we discuss the {\em static inter-cluster common influence} case, that is
${\bf\mathcal A}_{2}^{*}$: There exists a
 constant $\mathbb R^{K,K}$ stochastic matrix $B=[B_{p,q}]_{p,q=1}^{K}$, such that
\begin{eqnarray}
 \sum_{j\in C_{q}}L_{ij}(t)=B_{p,q},~i\in\mathcal C_{p},~p,q=\onetoK\label{B1a}
\end{eqnarray}
\begin{corollary}
\label{corollary2}
Suppose $L(t)$ satisfies the assumptions $\mathcal A_1,
\mathcal A_{2}^{*}, \mathcal A_4$ and $I(t)$ satisfies assumption $\mathcal
A_3$ and (\ref{special-input}).  If $u(t)$ does not converge to zero, then for almost all initials
$x(t_0)$ and $\tilde{\varsigma}$, the solution of system
(\ref{time-varying1}) can cluster consensus.
\end{corollary}
\begin{proof} Note that $e^{B(t-t_0)}$ is the solution matrix of
$\dot{z}(t)=B z(t)$. According to Corollary \ref{corollary1}, we only need to prove
$\limsup_{t\to\infty}{~\rm rank}(\int_{t_0}^te^{B(t-s)}u(s)ds)=K$.
Suppose the eigenvalues of $B$ are $\mu_{1},\cdots,\mu_{K}$ (possibly overlap), then the
eigenvalues of $W_2(t)$ should be $F_{i}(t)=\int_{t_{0}}^{t}e^{\mu_{i}(t-s)}u(s)ds,~ i=1,\cdots,K$. 
From the assumptions, $u(t)$ should be positive and negative intermittently with respect to time. Hence, there exists $\{\hat{t}_n\}_{n=1}^{\infty}$ such that
$\lim_{n\rightarrow\infty}F_{i}(\hat{t}_n)=F_1^*\neq 0,i=1,\cdots,K$.
\end{proof}
\begin{remark}
In Corollary 2, the assumption of existence of a static inter-cluster
common influence matrix $B$ can be weaken to be in the form of $a(t)B$,
with a scalar function $a(t)$. The sufficient condition can be easily
derived from the above analysis.
\end{remark}


\begin{remark}
The realization of the inputs $I_{i}(t)$ is technical: First, to realize
inter-cluster separation, $I_{i}(t)$ cannot converge to zero
asymptotically; otherwise, its influence to the system could disappear;
Second,  $\int_{t_0}^t I_i(s)ds$ should be bounded to guarantee boundedness
of  the system, which implies that $I_{i}(t)$ should be positive and negative
intermittently with respect to time, which results in the algebraic
difference (without absolute values) between the states in different
clusters is positive and negative intermittently as well. In particular, it can
be proved that the inter-cluster absolute difference has infinite zeros,
which implies that the algebraic values cross zeros infinitely (the proof
has not been shown in this paper due to the space limit). 
For example,
$I_{i}(t)=\alpha_{i}\sin(t)$ in the following.



\end{remark}
\section{Simulations}

In this section, two numerical simulations are provided to illustrate the
validity of the proposed theoretic results. The graph models considered
here come from \cite{Lu..}. We consider two time-varying graph models: one
is so called \emph{$p$-nearest- neighborhood regular graph}. The graph has
$N$ nodes, ordered by $\{1,\cdots,N\}$. Each node $i$ has $2r$
neighbors:$\{(i+j)\mod N: j=\pm1,\cdots,\pm r\}$, where mod denotes modular
operator. The nodes are divided into $K$ groups: $\mathcal{C}_{k}=\{i:i
\mod K=k\}, k=0,\cdots, K-1$, where $N \mod K=0$. The other one is
\emph{bipartite random graph}. $N$ (an even integer) nodes are divided into
two groups and each group has $N/2$ nodes. Each node has $m$ neighbors,
among which there are $s<m$ neighbors in the same group and the remaining
in another group. The neighbors are chosen with equal probability.

In these two examples, nodes are divided into two clusters, colored by red and blue respectively.
The non-identical inputs are defined as :
\begin{eqnarray*}
I_{p}(t)=\alpha_{p}sin(t), p=1,2.
\end{eqnarray*}
corresponding to each group with $\alpha_{1},\alpha_{2}$ are randomly
selected in [0,10] with the uniform distribution. Intra-cluster
synchronization is measured by difference of states in same clusters:
\begin{eqnarray*}
\Delta_{\mathcal{C}}(x(t))=\max_{p}\max_{i,i'\in\mathcal{C}_{p}}|x_{i}(t)-x_{i'}(t)|
\end{eqnarray*}
Inter-cluster separation is measured by $\eta_c(x(t))$ defined in (\ref{eta}).

Realize these two graph models respectively. We take a switching time
sequence $\{t_{k}\}_{k=0}^{+\infty}$ as a partition of $[0,+\infty)$ with
$0=t_{0}<t_{1}<\cdots$. Denote $\Delta t_{i}=t_{i}-t_{i-1}$, and the
switching time interval $\Delta t_{i}$ is uniformly distributed on $(0,1)$.

At every switching time, the graph topology stochastically choose from these two topologies
given in the top panels of Figs. 1 (a) and (b) respectively. For $t\in [t_{k-1},t_{k})$,
take $L_{ij}(t)=\sin(\frac{\pi (t-t_{k-1})}{\Delta t_{k}})$ if $j$ is a neighbor of $i$;
otherwise, $L_{ij}(t)=0$ and $L_{ii}(t)=-\sum_{j\neq i}L_{ij}(t)$.
Pick $\delta=1$. $L(t)$ has $\delta$-cluster-spanning-trees across $[t_i,t_{i+3})$.
Furthermore, the input $u(t)=\sin(t)$ and its integral are both bounded.
Meanwhile, we notice that the inter-cluster common influence matrix satisfies:
$B(t)=sin(\frac{\pi(t-t_{k-1})}{\Delta t_{k}})B$ when $t_{k-1}\le t<t_{k}$.
Denote $B(t)=b(t)B$. $\Psi(t,t_0)=e^{\int_{t_0}^tb(s)ds B}$ is the solution
matrix of system $\dot{z}(t)=B(t)z(t)$. 

Therefore, all conditions in Theorems 1 and 2 are satisfied. Choose the initial
values randomly. In Fig.1(a) and (b), the dynamical behaviors of the states
are plotted, while nodes in the same clusters are plotted in same color. In
the bottom panels of Fig.1 (a) and (b) , the blue, red and green curves
respectively show the dynamical behaviors of $\eta_{c}(x(t))$,
$\Delta_{c}(x(t))$ and $\eta_c(x(t))+\eta_c(v(t))$ with respect to the
time-varying topologies, where $v(t)\triangleq \dot{x}(t)$. All of them
show that the cluster consensus is reached. Please note that according to the arguments before, $I_{p}(t)=\alpha_{p}\sin(t)$ takes negative and positive
values intermittently so that $\int_{t_{0}}^{t}I_{i}(s)ds$ is bounded with
respect to $t$, but never converges to zero. This implies that there are infinite zeros of $\eta_{c}$ since its algebraic values
cross zeros infinite times, as shown in the third panels of Fig 1 (a,b) respectively.


\begin{figure}[!t]
\centering
\subfigure[Network topologies varying from $p$-nearest- neighborhood regular graph models.]
{
\includegraphics[height=.55\textwidth,width=.45\textwidth]
{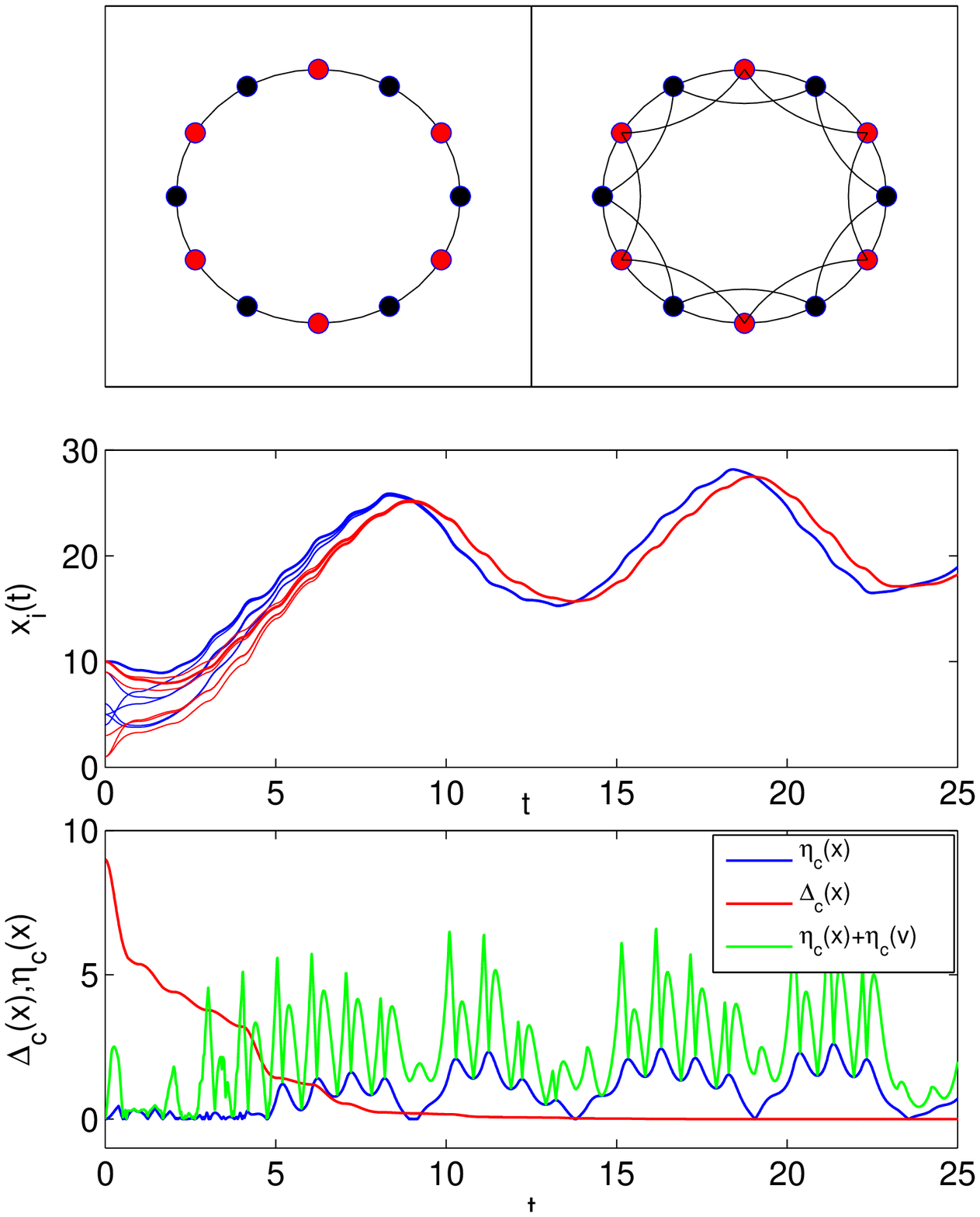}
}
\subfigure[Network topologies varying from bipartite random graph models.]{
\includegraphics[height=.55\textwidth,width=.45\textwidth]
{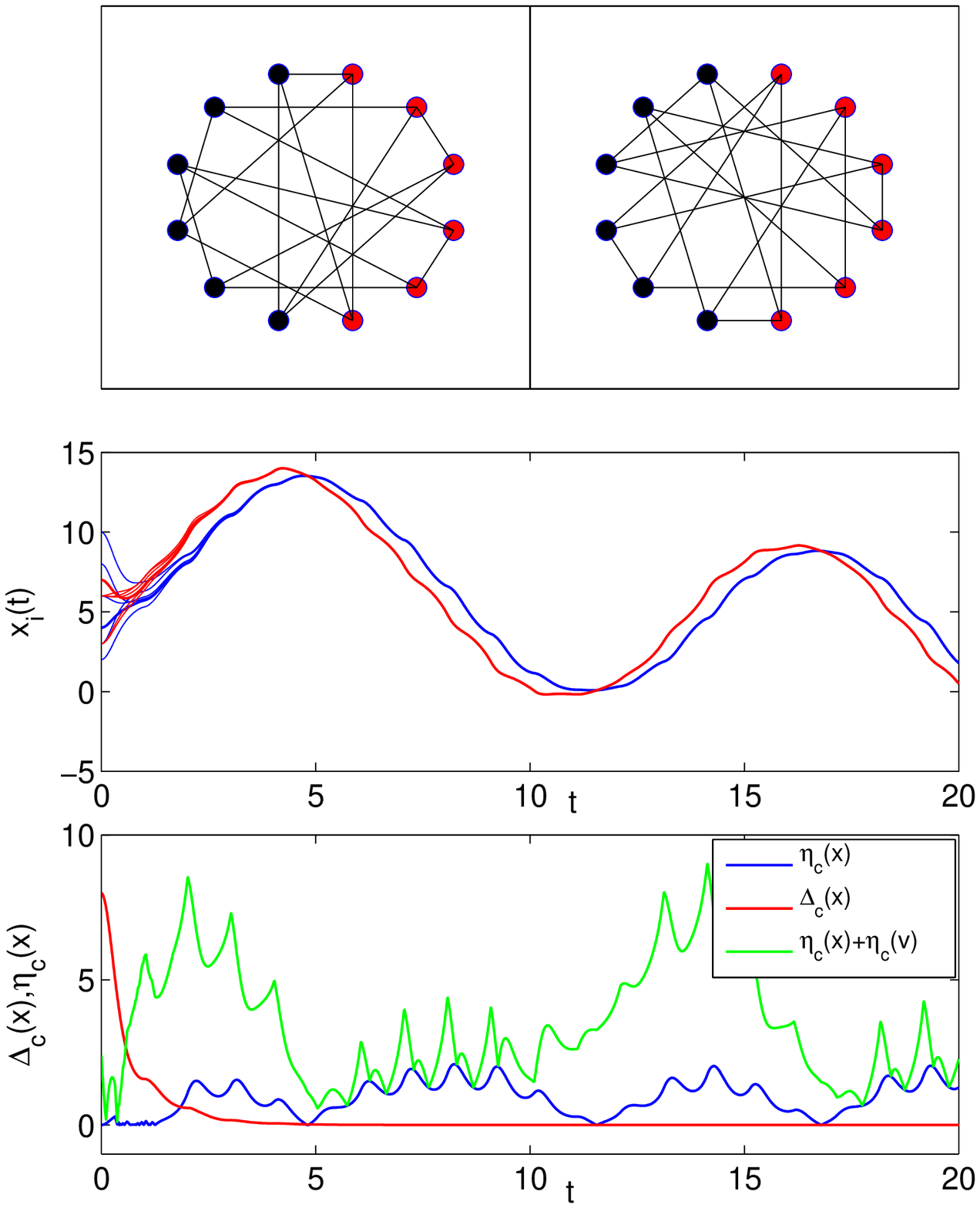}
}
\caption{The dynamics of states $\{x_i(t)\}$ and measures
$\Delta_{c}(x(t)),\eta_{c}(x(t))$. Red and blue nodes show the two clusters
of nodes respectively. }
\end{figure}


\section{Conclusions}

In this paper, we have investigated cluster consensus problem in
continuous-time networks of multi-agents with non-identical inter-cluster
inputs. Sufficient conditions for cluster consensus for systems with
time-varying graph topologies were derived. By defining cluster consensus
subspace, cluster consensus problem was transformed to the stability  of
the cluster consensus subspace under inter-cluster common influence
condition. The separation among states in different clusters were
guaranteed by external inputs. From algebraic graph theory, it was
indicated that the receiving same amount of information for agents in the
same cluster is a doorsill for the complete synchronization of agents in
the same cluster. The effectiveness of the proposed theoretical results
were demonstrated by numerical simulations.

\section{Appendix}
\textbf{Proof of Claim 1}: Define a $\mathbb R^{n,n}$ nonsingular matrix $P=[P_{1},\cdots,P_{n}]$ with the first $K$ column vectors composed of a basis of $\mathcal S_{\mathcal C}$. Thus, let
\begin{align*}
&\hat{L}(t)\triangleq P^{-1}L(t)P=\left[\begin{array}{cc}
B(t)&\hat{L}_{1,2}(t)\\
O&\hat{L}_{2,2}(t)
\end{array}\right],\\
&\hat{\Phi}(t,t_0)\triangleq P^{-1}\Phi(t,t_0)P=\left[\begin{array}{cc}
\Psi(t,t_0)&\hat{\Phi}_{1,2}(t,t_0)\\
O&\hat{\Phi}_{2,2}(t,t_0)
\end{array}\right],
\end{align*}
where $\Psi(t,t_0)$ is the solution matrix of system $\dot{x}(t)=B(t)x(t)$.
We define the {\em projection radius} (w.r.t. $\mathcal C$) of $\Phi(t,t_{0})$ as follows:
\begin{eqnarray*}
\rho_{\mathcal C}(\Phi(\cdot,t_{0}))=
\overline{\lim_{t\to\infty}}\bigg\{\|\hat{\Phi}_{2,2}(t,t_{0})\|\bigg\}^{1/t}
\end{eqnarray*}
and the {\em cluster Hajnal diameter} (w.r.t. $\mathcal C$) of $\Phi(t,t_{0})$ as follows:
\begin{eqnarray*}
\Delta_{\mathcal C}(\Phi(\cdot,t_{0}))=\overline{\lim_{t\to\infty}}\bigg\{\Delta_{\mathcal
C}(\Phi(t,t_{0}))\bigg\}^{1/t}
\end{eqnarray*}
for some norm $\|\cdot\|$ that is induced by vector norm.
Select one single row in $\Phi(t,t_0)$ from each cluster and compose these rows into a matrix,
denoted by H. Let $G=[P_1,\cdots,P_K]$. It can be seen that the rows of $GH$ corresponding to
the same cluster are identical. Then, we have
\begin{eqnarray*}\label{projection}
\|\Phi(t,t_0)-GH\|&=\|P^{-1}\Phi(t,t_{0})P-\left[\begin{array}{c}
E\\
O\end{array}\right]HP\|\\& =\|\left[\begin{array}{cc}
Y&Z\\
O&\hat{\Phi}_{22}(t,t_0)\end{array}\right]\|,
\end{eqnarray*}
which implies $\rho_{\mathcal C}(\Phi(\cdot,t_{0}))\leq\Delta_{\mathcal C}(\Phi(\cdot,t_{0}))$.
In Theorem 1, $\Delta_{\mathcal C}(\Phi(\cdot,t_{0}))<1$ has been proved.
Thus, $\rho_{\mathcal C}(\Phi(\cdot,t_{0}))<1$, which means
$\hat{\Phi}_{2,2}(t,t_0)$ converges to zero matrix exponentially.

It can be seen that $\hat{\Phi}(t,t_0)$ is the solution matrix of system
$
\dot{w}(t)=P^{-1}L(t)P w(t).
$
Consider the block form of vector $w(t)=\hat{\Phi}(t,t_{0})w(t_{0})$:
\begin{eqnarray}\label{component-wise}
\begin{cases}
w_{1}(t)=\Psi(t,t_{0})w_{1}(t_{0})+
\hat{\Phi}_{1,2}(t,t_{0})w_{2}(t_{0})\\
w_{2}(t)=\hat{\Phi}_{2,2}(t,t_{0})w_{2}(t_{0}).
\end{cases}
\end{eqnarray}
$\rho_{\mathcal C}(\Phi(\cdot,t_{0}))<1$ implies that $w_{2}(t)$ converges
to $\bold{0}$ exponentially. Then define the operators
$R_1=\lim_{t\to\infty}\Psi^{-1}(t,t_0)\hat{\Phi}_{1,2}(t,t_{0})$. It can be
verified that $R_1$ is well defined. Consider a subspace of $\mathbb
R^{n}$: $\tilde{V}=\bigg\{[z^{\top},v^{\top}]^{\top}\in\mathbb
R^{n}:~z=-R_1v\bigg\}.$

For any $n$-dimensional vector $w_{0}=[z_{0},v_{0}]^{\top}$, we rewrite
$w_{0}$ as a sum of $w_{0}^{1}+w_{0}^{2}$ with
$w_{0}^{1}=[z_{0}^{1},\bold{0}]^{\top}$,
$w_{0}^{2}=[z_{0}^{2},v_{0}]^{\top}$. If we take $w(t_{0})=w_{0}^{2}$ and
pick $z_{0}^{2}$ such that $w_{0}^{2}\in\tilde{V}$, then $w(t)$ converges
to $\bf{0}$ exponentially. That is, $PQw_0^2\in V$. On the other hand,
$PQw_0^1$ corresponds a vector in $\mathcal S_{\mathcal{C}}$. Therefore,
for any $n$-dimensional vector $x_{0}$, we can find $w_{0}$, such that
$x_{0}=PQw_{0}=PQw_{0}^{1}+PQw_{0}^{2}\in \mathcal{S}_{\mathcal{C}}+V$.

\section*{Acknowledgement}

The authors are very grateful to reviews for their useful comments and suggestions.

\end{document}